\documentclass[11pt,reqno,oneside]{amsart}

\usepackage{graphicx}
\usepackage{amssymb}
\usepackage{epstopdf}

\usepackage[a4paper, total={6in, 9.1in}]{geometry}

\usepackage{amsmath,amsfonts,amsthm,mathrsfs,amssymb,cite}
\usepackage[usenames]{color}

\newtheorem{theorem}{Theorem}[section]
\newtheorem{lemma}[theorem]{Lemma}
\newtheorem{proposition}[theorem]{Proposition}
\newtheorem{corollary}[theorem]{Corollary}
\theoremstyle{definition}
\newtheorem{definition}[theorem]{Definition}

\theoremstyle{remark}
\newtheorem{remark}[theorem]{Remark}

\numberwithin{equation}{section}

\newcounter{saveeqn}

\newcommand{\be}{\begin{equation}}
\newcommand{\ee}{\end{equation}}
\newcommand{\ba}{\begin{array}}
\newcommand{\ea}{\end{array}}

\newcommand{\bea}{\begin{eqnarray*}}
\newcommand{\eea}{\end{eqnarray*}}
\newcommand{\bean}{\begin{eqnarray}}
\newcommand{\eean}{\end{eqnarray}}

\def\B{\mathcal{B}}
\def\R{\mathbb{R}}
\def\C{\mathbb{C}}

\def\N{\mathbb{N}}
\def\h{\mathfrak{h}}

\def\n{\nabla}

\def\a{\alpha}
\def\b{\beta}
\def\d{\delta}
\def\e{\varepsilon}

\def\t{\theta}

\def\p{\partial}
\def\O{\Omega}
\def\ds{\displaystyle}

\title[Stability for Calder\'on's Problem by A Single Measurement]{Stable determination of polygonal inclusions in Calder\'on's Problem by A Single Partial Boundary Measurement}

\author{Hongyu Liu}
\address{Department of Mathematics, Hong Kong Baptist University, Kowloon, Hong Kong SAR, China}
\email{hongyu.liuip@gmail.com, hongyuliu@hkbu.edu.hk}

\author{Chun-Hsiang Tsou}
\address{Department of Mathematics, Hong Kong Baptist University, Kowloon, Hong Kong SAR, China}
\email{c-h\_tsou@hkbu.edu.hk}

\date{} 
\begin{document}
\maketitle

\begin{abstract}

We are concerned with the Calder\'on problem of determining the unknown conductivity of a body from the associated boundary measurement. We establish a logarithmic type stability estimate in terms of the Hausdorff distance in determining the support of a convex polygonal inclusion by a single partial boundary measurement. We also derive the uniqueness result in a more general scenario where the conductivities are piecewise constants supported in a nested polygonal geometry. Our methods in establishing the stability and uniqueness results have a significant technical initiative and a strong potential to apply to other inverse boundary value problems. 

\medskip



\end{abstract}

\section{Introduction}

In this paper, we consider the Calder\'on problem in determining the conductivity of a body from the associated boundary measurements of input-output pairs of electric current and voltage. The problem was first proposed and studied by A. P. Calder\'on in 1980 \cite{CALDERON2006} and has a profound impact on the field of inverse problems for partial differential equations (PDEs). It is also known as the inverse conductivity problem and has many practical applications including the geophysical prospecting and Electrical Impedance Tomography in medical imaging. 

The mathematical setup of the Calder\'on problem is described as follows. Let $\Omega\subset\mathbb{R}^n$, $n\geq 2$, be a bounded Lipschitz domain and $\gamma\in L^\infty(\Omega)$ be a positive function. Consider the following elliptic PDE problem for $u\in H^1(\Omega)$,
\begin{equation}\label{EQ_Calderon}
\begin{cases}
\displaystyle{\mathrm{div}(\gamma\nabla u)=0}\qquad & \mbox{in}\ \ \Omega,\medskip\\
\displaystyle{u=\psi\in H^{1/2}(\partial\Omega)} & \mbox{on}\ \ \partial\Omega. 
\end{cases}
\end{equation} 
Associated with \eqref{EQ_Calderon}, we introduce the following Dirichlet-to-Neumann (DtN) map, $\Lambda_\gamma: H^{1/2}(\partial\Omega)\mapsto H^{-1/2}(\partial\Omega)$, 
\begin{equation}\label{eq:DtN}
\Lambda_\gamma(\psi)=\gamma\partial_\nu u|_{\partial\Omega},
\end{equation}
where $u\in H^1(\Omega)$ is the solution to \eqref{EQ_Calderon} and $\nu\in\mathbb{S}^{n-1}$ is the exterior unit normal vector to $\partial\Omega$. In the physical situation, $\gamma$ signifies the conductivity of the body 
$\Omega$, $\psi$ is the input electric voltage, and $u$ and $\Lambda_\gamma(\psi)$ are respectively the induced electric potential in $\Omega$ and current on $\partial\Omega$. The Calder\'on's inverse problem is to recover $
\gamma$ by knowledge of $\Lambda_\gamma$, namely,
\begin{equation}\label{eq:ip1}
\Lambda_\gamma\rightarrow\gamma. 
\end{equation}
It is remarked that knowing the DtN map $\Lambda_\gamma$ means that the boundary Cauchy data $(\psi, \Lambda_\gamma(\psi))$ are known for any $\psi\in H^{-1/2}(\partial\Omega)$. Hence, if the DtN map is used, then infinite
boundary measurements are employed for the inverse problem. On the other hand, if only a single pair of Cauchy data is used, namely $(\psi, \Lambda_\gamma(\psi))$ for a fixed $\psi\in H^{1/2}(\partial\Omega)$, this is referred to 
as a single boundary measurement. Another scenario of practical importance is the so-called partial data problem. Let $\Gamma_0\Subset\partial\Omega$ be a proper subset of $\partial\Omega$. The Cauchy data $(\psi, \gamma\partial_\nu u|_{\Gamma_0})$, where $\mathrm{supp}(\psi)\subset\Gamma_0$, are referred to as a partial boundary measurement. In this paper, we study the Calderon problem in determining the unknown conductivity 
by a single partial boundary measurement. 

We are mainly concerned with the uniqueness and stability issues for the inverse problem, which are respectively concerned with establishing the results that $\Lambda_{\gamma_1}=\Lambda_{\gamma_2}$ if and only if
$\gamma_1=\gamma_2$, and $\|\gamma_1-\gamma_2\|_1\leq \tau(\|\Lambda_{\gamma_1}-\Lambda_{\gamma_2}\|_2)$, where $\|\cdot\|_j, j=1, 2$, are certain suitable norms and $\tau: \mathbb{R}_+\mapsto \mathbb{R}_+$ is the stability
function satisfying $\lim_{t\rightarrow+0}\tau(t)=0$. Sylvester and Uhlmann \cite{SylvesterUhlmann1987} proved the uniqueness result for $C^2$-smooth conductivity functions in dimension 3 using DtN map. Astala and P{\"a}iv{\"a}rinta \cite{astala2006calderon} generalized the uniqueness result in dimension $2$ for $L^\infty$ conductivity functions. The general stability estimate is firstly proven by Alessendrini \cite{alessandrini1988stable} with a logarithmic type estimation. Nachman \cite{Nachman1988} derived the uniqueness result in
dimension $2$ and a reconstruction algorithm to recover $\gamma$ from the DtN map $\Lambda_\gamma$. The partial-data Calder\'on problem using infinitely many boundary measurements are considered in \cite{CSU} and \cite{IUY}, respectively in three and two dimensions.

In many physical and engineering applications such as detecting a default in a composite medium, the conductivity coefficient is usually modelled as a piecewise constant function instead of a generic variable function. The inverse conductivity problem in such a special case is referred to as the inverse inclusion problem. The main challenge then becomes recovering the geometrical shape of an inclusion $D \Subset \O$. In the simplest case, the conductivity function $\gamma$ is usually assumed to take the following form 
\begin{equation}\label{eq:form}
\gamma=1+(k-1)\chi_D\quad\mbox{in}\ \ \Omega,
\end{equation}
where $k\in\mathbb{R}_+$ and $\chi_D$ signifies the characteristic function of the inclusion $D$. For instance, Ammari and Kang \cite{ammari2004reconstruction} use the generalized polarization tensors deduced from layer potential techniques to extract the geometrical characterizations of inclusions. Another recent progress is the recovery using the so-called multifrequency measurements in \cite{ammari2017identification,moi2,moi3}, where it is assumed that the conductivity inside the inclusion depends on the frequency of the time-dependent input current that follows form a physical model. In \cite{AMR}, by using infinitely many partial boundary measurements, the uniqueness is established in determining (possibly anisotropic) piecewise-constant conductivities. 

From a practical viewpoint, it is rather unrealistic to make use of infinitely many measurements. Hence, it is of practical significance weather a single boundary measurement is sufficient to recover the inverse inclusion problem. 
There are some existing results in the literature. Fabes, Kang and Seo \cite{fabes1999inverse} proved the H\"{o}lder type stability estimate when $D$ is a disc in $\R^2$. Triki and Tsou \cite{moi1} improved slightly this stability estimate and proposed a numerical scheme to recover the target disk under a single measurement. In the case that $D$ is a ball in $\R^3$, the uniqueness is obtained by Kang and Seo \cite{Kang_ball_R3}. {In the case that $D$ is a convex polygon in $\R^2$ or a convex polyhedron in $\R^3$, the uniqueness results are respectively given by Friedman, Isakov \cite{friedman1989uniqueness} and by Barcel\'{o}, Fabes and Seo \cite{barcelo1994inverse}. The reconstruction of an insulating curvilinear polygonal inclusion was considered in \cite{CHL}. On the other hand, Beretta, Francini and Vesella \cite{beretta2019lipschitz} proved recently the Lipschitz stability estimate using DtN map, namely infinitely many measurements, in the case of polygonal inclusions. We emphasise that in the aforementioned studies by a single measurement, it is a technical requirement that the content of the inclusion has to be uniform; that is, the conductivity of the inclusion is a positive constant; see Remark~\ref{rem:2} in what follows for more relevant discussion. To our best knowledge, there are lots of rooms to investigate on the inverse inclusion problem by a single measurement; for example, even for the case that the inclusion is supported in an ellipse, the uniqueness is still open.}

In this paper, we make a significant progress towards solving the inverse inclusion by a single measurement. We establish a logarithmic type stability estimate in determining the support of a convex polygonal inclusion by a single partial boundary measurement. We also derive the uniqueness result in a much more general scenario where the conductivities could be piecewise constants supported in a nested polygonal geometry. Indeed, both the uniqueness and stability results can be extended to cover more general scenarios; see Remarks~\ref{rem:1}--\ref{rem:3} for more relevant discussion. Our methods in establishing the stability and uniqueness results are inspired by a recent paper \cite{Liu_stability_polygon_helmholtz}, where the stability estimate was established in determining the convex polyhedral support of a medium function by a single acoustic far-field measurement. The mathematical argument in \cite{Liu_stability_polygon_helmholtz} was extended and developed from that in \cite{blaasten2014corners} for acoustic scattering from corner singularities. The key idea is to characterize the singularities of solutions in the phase space by making use of tools from microlocal analysis. In \cite{blaasten2014corners,Liu_stability_polygon_helmholtz}, the medium function appears in the lower-order term of the governing PDE and one has $H^2$-regularity of the underlying solution, which plays a critical role in the arguments therein. However, in the current setup for the Calder\'on problem, the conductivity appears in the leading-order term of the governing PDE and one has at most $H^1$-regularity of the solution. This difference brings significant challenges to our study.  Indeed, the gradient of the solution may blow up at near the corner. To overcome this difficulty, we make use of the decomposition of solutions to elliptic equations in polygonal domains \cite{Grisvard,Kozlov}. The introduction of this new ingredient together with a delicate balancing of a microlocal-type argument enables us to establish the desired uniqueness and stability results in $\mathbb{R}^2$.  We believe that our methods have a significant technical initiative and a strong potential to apply to other inverse boundary value problems including the three-dimensional extension and the case with anisotropic conductivities. 

The rest of the paper is organized as follows. In Section 2, we present and discuss the main stability result. Sections 3, 4 and 5 are devoted to the proof of the stability estimate. 
In Section \ref{Propagation}, we prove the propagation of smallness from the boundary to the inclusion. In Section \ref{Integral}, we introduce the integral identity and the construction of a complex geometrical solution, then use both them to prove the estimations on the solutions $u$, $u'$ respectively to the unperturbed conductivity equation and the perturbed one. We then prove Theorem \ref{main-theorem} in Section \ref{PROOF}. Finally, in Section 6, we present and prove the uniqueness result in a more general scenario.

\section{Statement of the main stability result}\label{MR}

Let $\Omega\subset\mathbb{R}^2$ be a bounded simply-connected domain with a Lipschitz boundary $\partial\Omega$ and let $D\Subset\Omega$ be a convex polygon. Consider a conductive inclusion of the form \eqref{eq:form} and the following conductivity equation for $u\in H^1(\Omega)$
\begin{equation}\label{EQ}
\begin{cases}
\displaystyle{\mathrm{div}\big((1+(k-1)\chi_D)\nabla u\big)=0}\quad &\mbox{in}\ \ \Omega,\medskip\\
\partial_\nu u=g\in H^{-1/2}(\partial\Omega)\quad&\mbox{on}\ \ \partial\Omega,\medskip\\
\displaystyle{\int_{\partial\Omega} u\, ds=0}. 
\end{cases}
\end{equation}
It is noted for the subsequent use that \eqref{EQ} is equivalent to the following system
\begin{equation}\label{EQ2}
\begin{cases}
\Delta u=0\quad &\mbox{in}\ \ \Omega\backslash\partial D,\medskip\\
u|_{\partial D}^+=u|_{\partial D}^-\quad &\mbox{on}\ \ \partial D,\medskip\\
\partial_\nu u|_{\partial D}^+=k\partial_\nu|_{\partial D}^-&\mbox{on}\ \ \partial D,\medskip\\
\partial_\nu u=g\in H^{-1/2}(\partial\Omega)\quad&\mbox{on}\ \ \partial\Omega,\medskip\\
\displaystyle{\int_{\partial\Omega} u\, ds=0},
\end{cases}
\end{equation}
where $\pm$ signify the limits taking from the outside and inside of $D$. The equations on $\p D$ in \eqref{EQ2} are called the jump relations. 

Next, we introduce some technical assumptions for our stability study. 
\begin{definition}\label{def:class}
Let $(D, k)$ be a conductive inclusion in $\mathbb{R}^2$ of the form \eqref{eq:form} and it is said to belong to the class $\mathcal{D}$ if the following conditions are fulfilled:
\begin{enumerate}
\item $D$ is a convex polygon and $k\in\mathbb{R}_+$ with $k\neq 1$ and $k\in (k_m, k_M)$, where $k_m$ and $k_M$ are two positive constants;
\item There exist $0<a_m<a_M<\pi$ such that the opening of the angle at each vertex of $D$ is in $(a_m,a_M)$;
\item The length of each edge of $D$ is at least $l>0$;
\item The distance of $D$ to the boundary of $\Omega$, namely $\mathrm{dist}(D,\p\O)$, is at least $\d_0>0$;
\item For any $D,D' \in \mathcal{D}$, the convex hull of $D\cup D'$ also has a distance at least $\d_0$ to the boundary $\p \O$.
\end{enumerate}
\end{definition}

The parameters $(a_m, a_M, l, \d_0)$ provide geometric characterizations of the polygonal inclusion. As a standard scenario in the stability estimate, the stability constant in our subsequent study shall depend on those geometric parameters. The next technical condition is about the input current $g$ on $\partial\Omega$. {It is required that the input $g\in H^{-1/2}(\partial\Omega)$ is such chosen that the induced electric potential in $\Omega$ satisfies for any vertex $x_c$ of $D$,
\begin{equation}\label{eq:condd1} 
\lim_{r\rightarrow +0}\displaystyle{ \frac{\int_{B_r(x_c)} |\nabla u|}{|B_r(x_c)|} }>0,
\end{equation}
where and also in what follows $B_r(x_c)$ signifies the disk centred at $x_c$ with radius $r\in\mathbb{R}_+$. In such a case, $g$ is referred to as admissible. Basically, \eqref{eq:condd1} means that $\nabla u$ is not vanishing at the vertex point. It is emphasized that the admissibility of boundary inputs for the conductivity problem has been studied in a different context \cite{gradu1,gradu2}. It is unobjectionable for us to assume throughout the rest of our study that the boundary input $g$ is always admissible. Let $\Gamma_0$ be an open and nonempty subset of $\partial\Omega$. $\Gamma_0$ signifies the measurement curve, that is $u|_{\Gamma_0}$ is the output for the inverse conductivity problem in our subsequent study. We would like to point out that it may occur that an admissible input $g$ satisfies $\mathrm{supp}(g)\subset\Gamma_0$. Hence, we actually consider a single partial boundary measurement in our study as long as the admissibility requirement is fulfilled. In what follows, $(a_m, a_M, l, \d_0), (k_m, k_M), g$ and $\Gamma_0$ are referred to as the a-priori parameters. It shall not be surprising to see that the stability constant in our quantitative estimate depends on the a-priori parameters. On the other hand, it is remarked at this point that in the qualitative uniqueness study in Section 6, only the admissibility condition is required. 
}

We are in a position to present our main stability result for the Calder\'on problem in determining polygonal inclusions. In what follows, for $D, D'\in\mathcal{D}$, we define
\[
d_\mathcal{H}(D, D')=\max \big(\sup_{x\in D}\mathrm{dist}(x,D'),\sup_{x'\in D'}\mathrm{dist}(x',D)\big),
\]
to be Hausdorff distance between $D$ and $D'$.  

\begin{theorem}\label{main-theorem}
Let $(D, k)$ and $(D',k')$ be two polygonal inclusions from the class $\mathcal{D}$. Let $u$ and $u'$ be the solutions to \eqref{EQ} associated with $(D, k)$ and $(D',k')$, respectively. Suppose that there holds
\begin{equation}\label{small}
\Vert u-u' \Vert_{H^{1/2}(\Gamma_0)} \leq \e.
\end{equation}
Then there exist constants $C,\b>0$ which depend only on the a-priori parameters such that when $\varepsilon\in\mathbb{R}_+$ is sufficiently small,
\begin{equation}\label{stability}
d_\mathcal{H}(D, D') \leq C (\ln|\ln\e|)^{-\b}.
\end{equation}
\end{theorem}

Three remarks of Theorem~\ref{main-theorem} are in order. 

\begin{remark}\label{rem:1}
The stability estimate in Theorem~\ref{main-theorem} in determining the support of the inclusion is clearly independent of its content. By taking $\varepsilon\rightarrow +0$, one readily has the uniqueness result; that is, if $u|_{\Gamma_0}=u'|_{\Gamma_0}$, then there holds $D=D'$. Furthermore, in Section 6, we shall show that there also holds $k=k'$. In fact, in Section 6, we show that the uniqueness holds in a much more general scenario with piecewise constant conductivities supported in a nested polygonal geometry.  
\end{remark}

\begin{remark}\label{rem:2}
The argument in proving Theorem~\ref{main-theorem} is of a ``localized" nature, which centers around a corner on $D\Delta D'=(D\backslash\overline{D'})\cup (D'\backslash\overline{D})$. Hence, there are two generalizations of Theorem~\ref{main-theorem} that can be made. The first one is that instead of requiring the conductivity in $D$ is a constant function, it is sufficient to require that the conductivity function is constant around each vertex of $D$ (may even take different values at different vertices), and it may be a variable function in the rest part of the inclusion. This is in sharp difference from the existing results in the literature as discussed in Section 1, where the content of the inclusion has to be uniform. The second one is that there might be multiple inclusions presented within the body $\Omega$; that is, the target inclusion is of the following form
\[
(D, k)=\bigcup_{j=1}^N (D_j, k_j),
\]
where each $(D_j, k_j)$ is a polygonal inclusion of the class $\mathcal{D}$. It is required that $D_j$, $j=1,2,\ldots N$ are pairwise disjoint and sparsely distributed. In principle, one can show those generalizations by a line-to-line copy of the proof of Theorem~\ref{main-theorem}. However, in order to have a rigorous and precise study, one still needs to derive the detailed geometrical and topological characterizations of the inclusions in these scenarios, which might be a bit tedious. In order to have a concise and clear exposition of the main idea of our study, we only consider the case in Theorem~\ref{main-theorem}. Nevertheless, the aforementioned two generalizations should be clear in the context of our study.  
\end{remark}

\begin{remark}\label{rem:3}
The Calder\'on problem is known to be severely ill-conditioned. Since we only make use of a single partial boundary measurement, the logarithmic-type estimate in \eqref{stability} is arguably optimal. It is worth further investigation on how the stability can be improved when more measurement data are available. 
\end{remark}

%

\section{Propagation of smallness from the boundary to the inclusion}\label{Propagation}
{
Starting from this section till to Section 5, we give the proof of Theorem~\ref{main-theorem}. In this section, we consider the propagation of smallness in \eqref{small} from the boundary to the inclusion. The main goal of this section is to estimate punctually $u-u'$ and $\n (u-u')$ in a neighborhood of a vertex $x_c$. The principal tool in this section is the standard procedure in \cite{Alessandrini} to establish the propagation of smallness in the studies of Cauchy problems. We will derive here similar estimations in $L^\infty$ norms and generalises them into the functions with a H\"older type singularity near $x_c$. As the price, our estimate of smallness depends on the distance to the convex polytope $Q'$ that we shall introduce later.}

To begin with, we introduce the three sphere inequality for harmonic functions in $L^{\infty}$ norms. The general three sphere inequality for any elliptic system is given by Alessandrini \cite{Alessandrini}, and a more precise estimate for harmonic functions is obtained by Korevaar and Meyers \cite{Meyers}. In what follows, for notational convenience, we set $B_R:=B_R(0)$. 

\begin{lemma}[Three-sphere Inequality, Meyers, 1994]
Let $0<r_1<r_2<r_3<R$, and $w\in H^1_{loc}(B_R)$ be a harmonic function in $B_R$. Then there exists $\tilde{\a}\in (0, 1)$, which depends only on ${r_2}/{r_1}$ and ${r_3}/{r_2}$, such that
\begin{equation}\label{3spheres}
\Vert w \Vert_{L^\infty(B_{r_2})} \leq \Vert w \Vert^{1-\tilde{\a}}_{L^\infty(B_{r_3})}\Vert w \Vert^{\tilde{\a}}_{L^\infty(B_{r_1})}.
\end{equation}
\end{lemma}
Starting from now on, for $r\in\mathbb{R}_+$, we set $r_1=r$, $r_2=2r$, $r_3=4r$ and consider $\tilde{\a}$ as a fixed value between $0$ and $1$.

\begin{lemma}\label{chain3spheres}
Let $U \subset \R^2$ be a bounded connected domain, and $\gamma \subset U$ be a rectifiable curve which links two distinct points $x,y \in U$ such that $B(\gamma,4r):=\cup_{x'\in \gamma}B_{4r}(x')\subset U$. Let $w \in L^\infty(U)$ be harmonic in $U$ and such that $\Vert w \Vert_{L^\infty(B_{4r}(x))}\leq 1$ and $\Vert w \Vert_{L^\infty(U)}=T\geq 1$.\\
Then there holds
\begin{equation}
\Vert w \Vert_{L^\infty(B_r(y))} \leq T \Vert w \Vert^{\tilde{\a}^{d_\gamma/r+1}}_{L^\infty(B_r(x))},
\end{equation}
where $d_\gamma$ is the length of the curve $\gamma$.
\end{lemma}
\begin{proof}
We construct a sequence of disks, each of radius $r$ and centred respectively at $x=x_1,x_2, \cdots, x_N,x_{N+1}=y$. Here $N=\lceil d_\gamma /r\rceil$. With this number of disks, it is possible to locate the centres $x_k \in \gamma$ such that $|x_{k+1}-x_k|\leq d_\gamma(x_k,x_{k+1})\leq r$. Here and also in what follows, $d_\gamma(x_k,x_{k+1})$ signifies the along the curve $\gamma$ between the two points $x_k$ and $x_{k+1}$. The latter inequality shows that $B_r(x_{k+1})\subset B_{2r}(x_k)$. It follows from the three-sphere inequality that
\[\Vert w \Vert_{L^\infty(B_r(x_{k+1}))} \leq T^{1-\tilde{\a}}\Vert w \Vert^{\tilde{\a}}_{L^\infty(B_r(x_k))}.\]
We apply successively the above inequality for each $k$, we thus have 
\[\Vert w \Vert_{L^\infty(B_r(y))}\leq T^{(1-\tilde{\a})(1+\tilde{\a}+\cdots+\tilde{\a}^{N-1})}\Vert w \Vert^{\tilde{\a}^N}_{L^\infty(B_r(x))}.\]
Using the facts that $1+\tilde{\a}+\cdots+\tilde{\a}^N \leq \frac{1}{1-\tilde{\a}}$, $T\geq 1$ and $N=\lceil d_\gamma /r\rceil\leq d_\gamma/r+1$, the claim follows.

The proof is complete. 
\end{proof}
{We state now the propagation of smallness in the interior. We introduce at first some geometric characterizations of the exterior part of inclusions. The set $\O'$ below is a connected subset of $\O\setminus (D\cup D')$ in the context of this paper.}

Let $\O' \subset \R^2$ be a bounded, connected and nonempty open domain of Lipschitz class. We define a subset $G_r$ of $\O'$ for $r\in\mathbb{R}_+$ as follows
\[
G_r:=\{x\in\O'|\text{ dist}(x,\p\O')>r\}.
\]
It is assumed that there exists $r_m\in\mathbb{R}_+$ such that $G_r$ is nonempty and connected for all $0<r \leq r_m$. It is also assumed that the periphery of $\O'$, denoted by $|\p \O'|$, is finite.

\begin{proposition}\label{propagation}
Let $w\in H^1_{loc}(\O')$ be a harmonic function in $\O'$. Suppose that $w\in L^\infty_{loc}(\O')$ with an $L^\infty$-norm in $G_r$ satisfying $\Vert w \Vert_{L^\infty(G_r)}=T_r\geq 1$. Assume that there exist $B_{r_0}(x_0)\subset \O'$ and $0<\epsilon<1$ such that
\[\Vert w \Vert_{L^\infty(B_{r_0}(x_0))} \leq \epsilon.\]
Then for $r<\frac{1}{5}\min(r_0,r_m)$ there holds
\begin{equation}
\Vert w \Vert_{L^\infty(\overline{G_{5r}})} \leq T_r \epsilon^{\tilde{\a}^{|\p \O'|/r+1}}. 
\end{equation}
\end{proposition}

\begin{proof}
As $r<r_0/5$, it is immediately seen that $x_0\in G_{5r}$. For any point $y \in G_{5r}$, as $G_{5r}$ is connected, there exists a rectifiable curve $\gamma \subset G_{5r}$ that links  $x_0$ to $y$. Moreover, $B(\gamma,4r) \subset B(G_{5r},4r)\subset G_r$. By applying Lemma \ref{chain3spheres}, the proposition readily follows from using the fact that there always exists a curve with a length at most $|\p\O'|$ which link any pair of points in $\O'$.

The proof is complete. 
\end{proof}
{ In the following proposition, we give an estimation of $u-u'$ and $\n (u-u')$ near a convex corner of $D$ or $D'$. We assume there exists a convex Lipschitz domain $Q$ such that $\p \O'= \p Q'\cup (\p \O' \setminus \p Q')$ and $\mathrm{dist}(\p Q', \p \O'\setminus \p Q') > \d_0$.}

\begin{proposition}\label{propagation_small}
Let $\O'$, $Q'$ be defined as above. Let $w \in H^1_{loc}(\O')$ be a harmonic function in $\O'$. Let $x_c \in \p Q'$, $P\in \N$, $0<\a<1$, we assume that the function $\tilde{w}_P: x \mapsto |x-x_c|^P w(x)$ is of class $\mathcal{C}^\a$ in $\overline{\O'}$ with a norm at most $T\geq 1$. We assume there exist $B_{r_0}(x_0)\subset \O'$ and $0<\epsilon<1$ such that
\[\Vert w \Vert_{L^\infty(B_{r_0}(x_0))} \leq \epsilon.\]
If 
\begin{equation}\label{critere_epsilon_h}
\epsilon <\epsilon_m=\left[\exp\exp \left( \frac{5|\p \O'||\ln \tilde{\a}|}{(1-\a)\min(r_0,r_m,\d_0)}\right) \right]^{-1},
\end{equation}
then there exists a constant $C_0>0$ depending only on $\a$, $\tilde{a}$, $P$ and $|\p\O'|$ such that,
\begin{equation}\label{estimation_exterieur}
|w(x)|\leq \frac{C_0T}{\mathrm{dist}(x,\p Q')^P}(\ln|\ln \epsilon|)^{-\a},
\end{equation}
for all $x \in\O'\cap B_{\d_0}(x_c)$.
\end{proposition}

\begin{proof}
Let $x\in \O'\cap B_{\d_0}(x_c)$ and we denote $r=\mathrm{dist}(x,\p Q')$. 

Define
\[\tilde{r}=\tilde{r}(\epsilon)=\frac{5|\p\O'||\ln \tilde{\a}|}{(1-\a)\ln|\ln \epsilon|}>0.\]
The upper bound (\ref{critere_epsilon_h}) implies $\tilde{r} < \min(r_0,r_m,\d_0)$.\par

We assume that $r \leq \tilde{r}$ now, then there exists $y\in \p Q'$ such that $|x-y|\leq \tilde{r}<\d_0$. By the convexity of $Q'$, there exist $x'\in \R^2\setminus Q'$ such that $\mathrm{dist}(x',\p Q')=|x'-y|=\tilde{r}< \d_0$. The condition $\mathrm{dist}(\p Q', \p \O'\setminus \p Q') > \d_0$ implies $x' \in \O'$. We have at the same time, $|x-x'|\leq |x-y|+|x'-y|\leq 2 \tilde{r}$. The conditions $\mathrm{dist}(x',\p Q')=\tilde{r}$ and $ \frac{\tilde{r}}{5}< \frac{1}{5}\min(r_0,r_m)$ guarantee that we can apply Proposition \ref{propagation}.\par

From the assumption on the function $\tilde{w}_P$, we have, $w\in L^\infty_{loc}(\O')$ and $\Vert w \Vert_{L^\infty(G_{\tilde{r}/5)})} =T_{\tilde{r}/5} \leq T (\frac{5}{\tilde{r}})^P$. Thus,
\begin{equation}\label{application_directe_propagation}
|w(x')|\leq T \left(\frac{5}{\tilde{r}}\right)^P\epsilon^{\tilde{\a}^{\frac{5|\p \O'|}{\tilde{r}}+1}}.
\end{equation}
By the H\"older continuity of the function $w_P$ and  by the facts $r=\mathrm{dist}(x,\p Q')\leq |x-x_c|$ and $ r < \tilde{r}$, it follows,
\begin{equation}\label{estimation_proche}
\begin{split}
  &|w(x)|\ds \leq \frac{1}{|x-x_c|^P}\left(\Vert w_P \Vert_{\mathcal{C}^\a(\overline{\O'})}|x-x'|^\a+|x'-x_c|^P|w(x')| \right) \\
 \ds \leq &\frac{1}{|x-x_c|^P}\left( T 2^\a \tilde{r}^\a +M_P (|x-x'|^P+|x-x_c|^P)T \left(\frac{5}{\tilde{r}}\right)^P\epsilon^{\tilde{\a}^{\frac{5|\p \O'|}{\tilde{r}}+1}} \right)  \\
 \ds \leq & \frac{T 2^\a \tilde{r}^\a}{|x-x_c|^P}+M_P T5^P\left(\frac{1}{\tilde{r}^P}+\frac{2^P}{|x-x_c|^P}\right)\epsilon^{\tilde{\a}^{\frac{5|\p \O'|}{\tilde{r}}+1}}  \\
\ds \leq & \frac{T}{r^P}\left( 2^\a\tilde{r}^\a+ M_P 5^P(1+2^P)\epsilon^{\tilde{\a}^{\frac{5|\p \O'|}{\tilde{r}}+1}}\right),
\end{split}
\end{equation}
where $M_P\geq 1$ is such that for every $a,b \geq 0$ one has $(a+b)^P\leq M_P(a^P+b^P)$.\par

The choice of $\tilde{r}(\epsilon)$ implies that
\begin{equation}\label{r_alpha}
\tilde{r}^\a=(\frac{5|\p\O'||\ln \tilde{\a}|}{(1-\a)})^\a (\ln|\ln \epsilon|)^{-\a}, \qquad \frac{5|\p \O'|}{\tilde{r}}=\frac{1-\a}{|\ln \tilde{a}|}\ln|\ln \epsilon|,
\end{equation}
and hence
\begin{equation}\label{epsilon_exponant}
\epsilon^{\tilde{\a}^{\frac{5|\p \O'|}{\tilde{r}}+1}}=e^{-|\ln \epsilon|\tilde{\a}^{\frac{5|\p \O'|}{\tilde{r}}+1}}=e^{-\tilde{\a}|\ln\epsilon|^{1-(1-\a)}}\leq \frac{1}{\tilde{\a}|\ln \epsilon|^\a}\leq \frac{1}{\tilde{\a}}(\ln|\ln\epsilon|)^{-\a}.
\end{equation}
Using (\ref{estimation_proche}), (\ref{r_alpha}) and (\ref{epsilon_exponant}) and setting
\begin{equation}\label{valeur_C}
C_0=\left( \frac{10|\p\O'||\ln \tilde{\a}|}{1-\a} \right)^\a+\frac{M_P5^P(1+2^P)}{\tilde{\a}},
\end{equation}
the claim follows in the case $r< \tilde{r}$.\par 
In the case $r\geq \tilde{r}$, it is sufficient to apply directly Proposition \ref{propagation} with $T_r=T(\frac{5}{r})^P$. Then using (\ref{epsilon_exponant}), we have 
\[
\epsilon^{\tilde{\a}^{\frac{5|\p \O'|}{r}+1}} \leq \epsilon^{\tilde{\a}^{\frac{5|\p \O'|}{\tilde{r}}+1}} \leq \frac{1}{\tilde{\a}}(\ln|\ln\epsilon|)^{-\a}.
\]
Therefore the claim follows.

The proof is complete. 
\end{proof}

\begin{corollary}\label{propagation_boundary}
Let $\O'$, $Q'$, $x_c$ and $w$ satisfy the same assumptions in Proposition \ref{propagation_small}. Let $\Gamma'_0 \subset \p \O'\setminus \p Q'$ be a nonempty open subset. We assume that 
\[\Vert w \Vert_{H^{1/2}(\Gamma'_0)}+\Vert \p_\nu w\Vert_{H^{-1/2}(\Gamma'_0)} \leq \epsilon.\]
Then, if $\epsilon<\epsilon_m(\Gamma'_0,|\p\O'|, r_m, \d_0, \tilde{\a},\a)$, we have
\begin{equation}
|w(x)|\leq \frac{C_0 T}{\mathrm{dist}(x,\p Q')^P}(\ln|\ln \epsilon|)^{-\a},
\end{equation}
for $x \in\O'\cap B_{\d_0}(x_c)$. Here the constant $C_0$ is given by (\ref{valeur_C}).
\end{corollary}

\begin{proof}
It follows from Lemma 6.1 and Theorem 6.2 in \cite{Alessandrini}, there exists a neighborhood of a point $P\in \Gamma'_0$, denoted by $\mathcal{V}$, such that $\mathcal{V}$ depends only on the geometric characteristics of $\Gamma'_0$ and $\Vert w \Vert_{H^1(\mathcal{V})}\leq c_1 \epsilon $. We can therefore choose $B_{r_0}(x_0)\subset\mathcal{V}$ with $r_0$ depending only on the geometrical characteristics of $\p \O'$. Then using the Sobolev embedding $H^1 {\hookrightarrow} L^\infty$ in $\R^2$, we have $\Vert w \Vert_{L^\infty(B_{r_0}(x_0))}\leq c_2\epsilon$. Thus , with a suitable choice of $\epsilon_m$, the claim follows by a straightforward application of Proposition \ref{propagation_small}.

The proof is complete. 
\end{proof}
\begin{remark}
In fact, using interior elliptic regularity estimate, $w$ is real analytic in a neighborhood of $\Gamma_0'$.
\end{remark}

\section{An Integral identity and several critical estimates}\label{Integral}


\begin{lemma}\label{lemme_geo}
Let $D,D'\subset \R^2$ be two open bounded convex polygons. Let $Q$ be the convex hull of $D\cup D'$. If $x_c$ is a vertex of $D$ such that $\mathrm{dist}(x_c,D')=\h$, where $\h$ gives the Hausdorff distance,
\begin{equation}\label{eq:haus}
\h=d_\mathcal{H}(D, D'),
\end{equation}
then $x_c$ is a vertex of $Q$. If the angle of $D$ at $x_c$ is $a$, then the angle of $Q$ at $x_c$ is at most $(a+\pi)/2<\pi$.
\end{lemma}

\begin{proof}
See the appendix in \cite{Liu_stability_polygon_helmholtz}.
\end{proof}
We assume from now on that $D\neq D'$. Let $x_c\in \overline{D}$ be a vertex in Lemma \ref{lemme_geo}. Then there exists $h\in\mathbb{R}_+$ such that $B_{h}(x_c)\cap D'=\emptyset$. Let $Q$ be the convex hull of $D\cup D'$ and $\B$ be the open disk $B_h(x_c)$. We denote respectively by $\tilde{D}$ and by $\tilde{Q}$ the sectors $\B\cap D$ and $\B\cap Q$. Let $b$ signify the opening of the angle of $Q$ at $x_c$. Then one has $a_m \leq b \leq (a_M+\pi)/2<\pi$ by Lemma \ref{lemme_geo}. We choose the polar coordinate system such that $x_c$ is the origin point and $\tilde{Q}$ coincides with the following sector,
\[\tilde{Q}=\{(r,\theta)|0<r< h,-b/2 < \theta < b/2\}.\]
We next define the integral contours on which we derive the estimates; see Figure \ref{coin} for a schematic illustration. Let
\begin{eqnarray}
\ds \Gamma^\pm:= \p D \cap \B,\nonumber\\
\ds \p S^i_D:=\p \B \cap D \subset \p S^i_Q,\nonumber\\
\ds \p S^i_Q:=\{(r,\theta)| r=h,-\frac{\pi+b}{4}\leq \theta\leq \frac{\pi+b}{4}\},\nonumber
\end{eqnarray}
and $\p S^e$ is a circular arc passing through the following three points in the polar coordinate: $(h,\frac{\pi+b}{4})$, $(\frac{1}{\tau},\pi)$, $(h,-\frac{\pi+b}{4})$ for a $\tau>0$. The idea of construction is to construct the contour $\p S^e$ such that $(x-x_c)\cdot \hat{x}\geq -\frac{1}{\tau}$ and $\mathrm{dist}(x,\p Q)\geq \frac{1}{2\tau}$ for all $x \in \p S^e$.
\begin{figure}[!ht]
\includegraphics[scale=2.2]{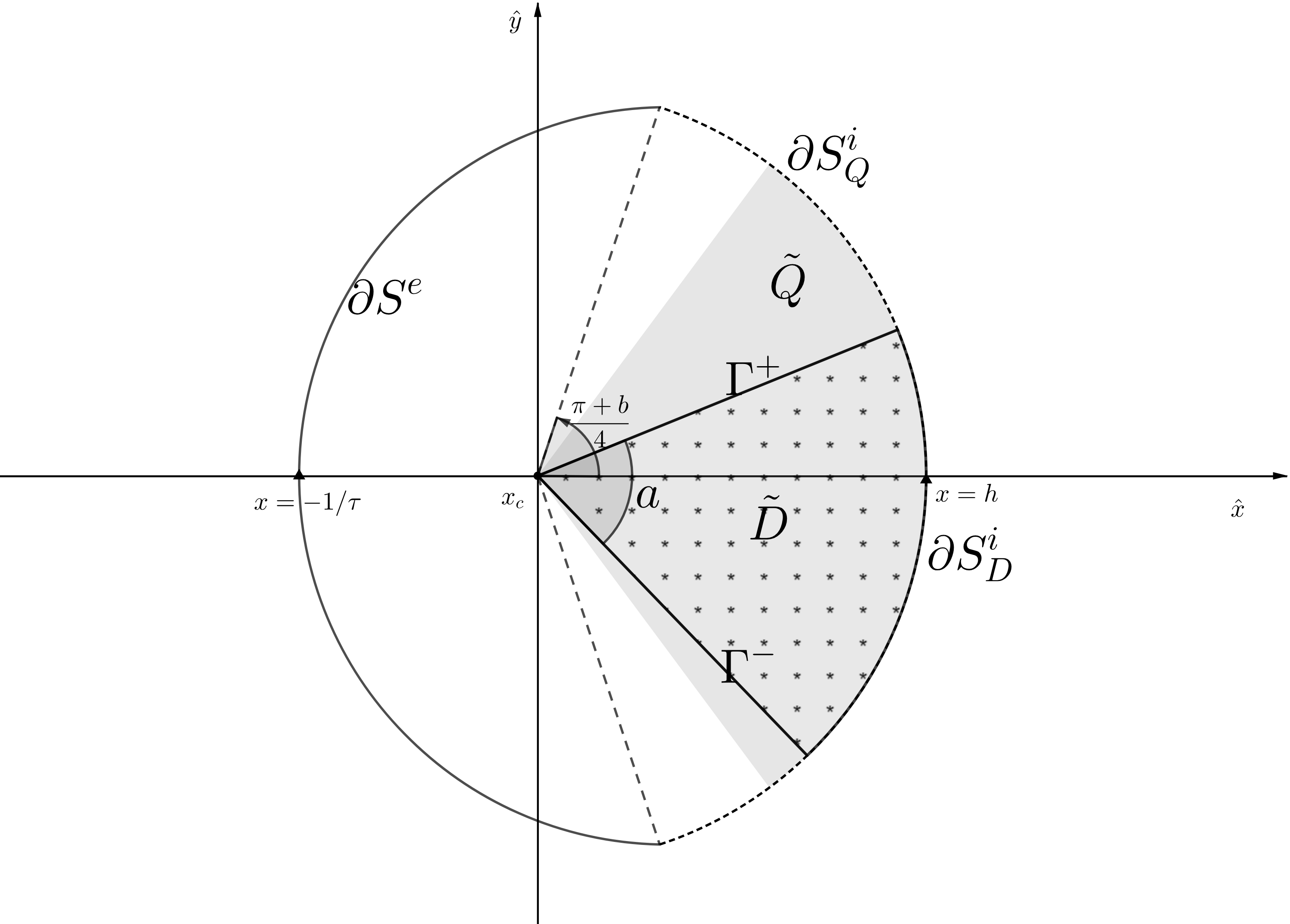}
\caption{The sectors $\tilde{D}$, $\tilde{Q}$ and the integral contours.}
\label{coin}
\end{figure}

\begin{proposition}
Let $u_0$ be a harmonic function in $\B$, then
\begin{equation}\label{IntId}
(k-1)\int_{\Gamma^\pm}u_0\p_\nu u d\sigma=\int_{\p S^i_Q \cup \p S^e} (u-u')\p_\nu u_0-u_0\p_\nu(u-u')  d\sigma,
\end{equation}
where the integral over $\Gamma^\pm$ is taken the values of $\p_\nu u$ in the interior of $\tilde{D}$.
\end{proposition}

\begin{proof}
This proposition follows directly from the jump relation and Green's formula.
\end{proof}

We define here a special type of harmonic functions which is the so-called complex geometric optics (CGO) solutions. In this paper, we define the CGO solution as follows. Let $\tau>0$, we choose $\rho=\rho(\tau):=\tau(-\hat{x}+i\hat{y})\in \C^2$. For all $x\in \R^2$,
\begin{equation}\label{CGO}
u_0(x)=e^{\rho \cdot (x-x_c)}.
\end{equation}
It is easy to check that $\rho \cdot \rho=0$ and thus $u_0$ is harmonic in $\R^2$.

\begin{proposition}\label{Propo_Upper}
Let $\tau>0$, $u_0$ be a CGO solution defined by (\ref{CGO}). We assume the solution $u$ admits the following composition in a neighborhood of $x_c$,
\begin{equation}\label{decom_sing}
u=u_{sing}\zeta+u_{reg}\ \ \text{with }\ u_{sing}(r,\t)=K r^\eta \phi(\t),
\end{equation}
where $K>0$, $0<\eta_m\leq\eta \leq_M<1$, $\phi$ is a piecewise smooth function on $[0,2\pi]$ and $\zeta$ is a cut-off function satisfying
\[
\zeta(r)=\begin{cases}
1\quad\mbox{when}\ \ r\leq \varrho,\\
0\quad\mbox{when}\ \ r\geq 2\varrho,
\end{cases}
\] 
We assume here $\varrho \geq h$ and the function $u_{reg}$ has a $H^2$ regularity in $\O$. \\
Then there hold
\begin{eqnarray}\label{IntId2}
\ds \int_{\Gamma^\pm_\infty}  u_0 \p_\nu u_{sing} d\sigma &=\ds\int_{\Gamma^\pm_\infty \setminus \Gamma^\pm}u_0 \p_\nu u_{sing} d\sigma+\int_{\p S^i_D}u_0\p_\nu u_{reg} d\sigma -\int_{\tilde{D}}\n u_0\n u_{reg} dx  \nonumber \\
 &\ds+\frac{1}{k-1}\int_{\p S^i_Q \cup \p S^e} (u-u')\p_\nu u_0-u_0\p_\nu(u-u') d\sigma,
\end{eqnarray}
and furthermore the following estimate,
\begin{eqnarray}\label{UpperBonds}
&\ds C|\int_{\Gamma^\pm_\infty}u_0 \p_\nu u_{sing} d\sigma |\leq K\tau^{-\eta}e^{-\a'\tau h/2}+\tau^{-1}\Vert u_{reg}\Vert_{H^2(\O)}+h e^{-\a'\tau h}\Vert u_{reg}\Vert_{H^2(\O)}\nonumber\\
& +h e^{-\a' \tau h}(\Vert \p_\nu (u-u') \Vert_{L^\infty(\p S^i_Q)}+\tau\Vert u-u' \Vert_{L^\infty(\p S^i_Q)})\nonumber\\
&+h(\Vert \p_\nu (u-u') \Vert_{L^\infty(\p S^e)}+\tau\Vert u-u' \Vert_{L^\infty(\p S^e)}),
\end{eqnarray}
where $\a'=\cos(\frac{\pi+b}{4})>0$, $C$ depends only on the parameters $\eta_m,a_m,a_M$ and {$\Gamma^\pm_\infty$ signify the two rays from the origin and extending the segments $\Gamma^\pm$ to infinity.}
\end{proposition}

\begin{proof}
It follows directly form (\ref{IntId}) and (\ref{decom_sing}),
\begin{eqnarray}
\ds \int_{\Gamma^\pm_\infty}  u_0 \p_\nu u_{sing} d\sigma &=\ds\int_{\Gamma^\pm_\infty \setminus \Gamma^\pm}u_0 \p_\nu u_{sing} d\sigma -\int_{\Gamma^\pm}u_0\p_\nu u_{reg} d\sigma \ds \nonumber \\
 &\ds+\frac{1}{k-1}\int_{\p S^i_Q \cup \p S^e} (u-u')\p_\nu u_0-u_0\p_\nu(u-u') d\sigma.
\end{eqnarray}
The equation (\ref{EQ}) implies that $u$ is harmonic in $\tilde{D}$. Then using the expression of $u_{sing}$, we have for $x=(r,\theta)\in \tilde{D}$,
\[\triangle u_{reg}(x)=-\triangle u_{sing}(x)=-Kr^{\eta-2}(\phi''(\theta)+\eta^2\phi(\theta)).\]
Using the fact $u_{reg}\in H^2(\O)$, it comes out $\phi''(\theta)+\eta^2\phi(\theta)=0$, which implies $\triangle u_{reg}=0$ in $\B \setminus \Gamma^\pm$ and $\phi(\theta)=K \cos(\eta \theta+\Phi)$. Then it follows by Green's formula,
\[\int_{\Gamma^\pm}u_0\p_\nu u_{reg} d\sigma=-\int_{\p S^i_D}u_0\p_\nu u_{reg} d\sigma +\int_{\tilde{D}}\n u_0\n u_{reg} dx.\]
Hence the integral identity (\ref{IntId2}) follows.

We are now about to estimate each terms in the right hand side of (\ref{IntId2}). We begin with introducing here the incomplete Gamma function $\Gamma(s,x):\R_+\times\R_+\rightarrow\R_+$, which is defined as
\[
\Gamma(s,x):=\int_x^{+\infty} t^{s-1}e^{-t} dt.
\]
Using the fact that $e^{-t}\leq e^{-x/2}e^{-t/2}$ for all $t\geq x$ and a single change of variable $t'=t/2$, we have the estimation 
\[
\Gamma(s,x)\leq 2^s\Gamma(s)e^{-x/2},
\] 
where $\Gamma(s)$ is the value on $s$ of the complete Gamma function.

Next we give the estimations corresponding to each integrals in the right hand side of (\ref{IntId2}). By the construction of the CGO solution, for all $x\in\tilde{D}\cup \Gamma^\pm \cup \p S^i_Q$,
\begin{equation}\label{control_u0}
|u_0(x)|\leq e^{\Re(\rho)\cdot(x-x_c)}\leq e^{-\a'\tau r}.
\end{equation}
On the other hand, the construction of the contour $\p S^e$ implies, for all $x\in \p S^e$,
\begin{equation}\label{control_u_0_Se}
|u_0(x)|\leq e^{\Re(\rho)\cdot(x-x_c)} \leq e^{-\tau\cdot \frac{1}{\tau}} = e^{-1}.
\end{equation}
The estimate of the first integral is therefore straightforward,
\begin{eqnarray}
&\ds\left|\int_{\Gamma^\pm_\infty \setminus \Gamma^\pm}u_0 \p_\nu u_{sing} d\sigma \right | & \ds\leq 2K \int_h^{+\infty}r^{\eta-1}e^{-\a'\tau r}dr \nonumber\\
& \leq 2K a'^{-\eta}\tau^{-\eta}\Gamma(\eta,\a'\tau h) &\leq C_1 K\tau^{-\eta}e^{-\a'\tau h/2},
\end{eqnarray}
where $C_1$ depends only on the a-priori parameters $\eta_m,a_m,a_M$.

Since $u_{reg}\in H^2(\O)$, we clearly have that $\n u_{reg}\in H^1(\O)$. Then using (\ref{control_u0}) and the Sobolev embedding $H^1\underset{continue}{\hookrightarrow} L^{\infty}$ in $\R^2$, it follows
\[
\left|\int_{\p S^i}u_0\p_\nu u_{reg} d\sigma\right| \leq C_2 h e^{-\a'\tau h}\Vert u_{reg}\Vert_{H^2(\O)},
\]
and
\[
\left|\int_{\tilde{D}}\n u_0\n u_{reg} dx\right|\leq C_3 (h e^{-\a'\tau h}+\frac{1}{\tau})\Vert u_{reg}\Vert_{H^2(\O)},
\]
where $C_2,C_3$ depend only on the opening $a$ of the angle of $D$ at $x_c$.\par
Using (\ref{control_u0}), (\ref{control_u_0_Se}) and direct estimates on the integrals, we can obtain the estimates on $u-u'$.

The proof is complete. 
\end{proof}

\begin{proposition}\label{Propo_Lower}
Let $u_0$, $u_{sing}$ be functions defined respectively by (\ref{CGO}) and (\ref{decom_sing}) with $0<\eta<1$ and $\phi$ be piecewise smooth on $[0,2\pi]$. Then there holds
\begin{equation}\label{LowerBound}
\left|\int_{\Gamma^\pm_\infty}u_0 \p_\nu u_{sing} d\sigma \right|=K \Gamma(\eta)\left|\phi'(\theta^+)e^{i a\eta}-\phi'(\theta^-)\right|\tau^{-\eta}\geq K \Gamma(\eta)\sin(a\eta)\tau^{-\eta},
\end{equation}
where $\theta^\pm$ signify the arguments of the vectors along $\Gamma^\pm$.
\end{proposition}

\begin{proof}
Let $z_0\in\C$ with $\Re(z_0)<0$. We define the integral 
\[
I(z_0,\eta):=\int_0^{+\infty} r^{\eta-1}e^{z_0 r} dr.
\]
Using the change of variables $z=z_0 r$, we have
\[I(z_0,\eta)=z_0^{-\eta}\int_{z_0\R_+} z^{\eta-1}e^z dz.\]
To calculate the complex integral above, we choose a contour $\gamma=\gamma_{z_0}\cup\gamma_A \cup\gamma_- \cup \gamma_\epsilon$, where the four portions are defined in the following way for $0<\epsilon<A$,
\[
\begin{split}
\gamma_{z_0}:=& \{z_0 t | \epsilon\leq t\leq A \},\\
\gamma_A:=& \{Ae^{i\theta}|\arg(z_0)\leq\theta\leq\pi\},r\\
\gamma_-:=& [-A,-\epsilon],\\
\gamma_\epsilon:=& \{\epsilon e^{i\theta}|\arg(z_0)\leq\theta\leq\pi\}.
\end{split}
\]
The function $z\mapsto z^{\eta-1}e^z$ is holomorphic in the interior domain defined by the contour $\gamma$ because we can choose the determination to the power function as the real positive axis. As an immediate consequence,
\[\int_\gamma z^{\eta-1}e^z dz=0.\]
Using the fact $\Re(z_0)<0$, one can easily obtain the following estimations,
\begin{eqnarray}
|\int_{\gamma_\epsilon} z^{\eta-1}e^z dz|\leq C\epsilon^\eta \underset{\epsilon \rightarrow 0}{\longrightarrow} 0, \nonumber\\
|\int_{\gamma_A} z^{\eta-1}e^z dz|\leq CA^\eta e^{A\Re(z_0)} \underset{A \rightarrow +\infty}{\longrightarrow} 0. \nonumber
\end{eqnarray}
Therefore, one can show that there holds
\begin{equation}\label{eq:11}
\begin{split}
& \int_{z_0\R_+} z^{\eta-1}e^z dz\\
=& \lim_{\epsilon \rightarrow 0,A\rightarrow +\infty} \int_{\gamma_{z_0}} z^{\eta-1}e^z dz\\
=& \lim_{\epsilon \rightarrow 0,A\rightarrow +\infty} -\int_{\gamma_A \cup \gamma_- \cup \gamma_\epsilon} z^{\eta-1}e^z dz\\
=& -\int_{-\infty}^0 t^{\eta-1}e^t dt=(-1)^\eta\int_0^{+\infty} t^{\eta-1}e^t dt =(-1)^\eta \Gamma(\eta).
\end{split}
\end{equation}
Thus,
\begin{equation}
I(z_0,\eta)=(\frac{-1}{z_0})^\eta \Gamma(\eta).
\end{equation}
On the other hand, it follows from the construction of the CGO solutions that $u_0(r,\theta^\pm)=e^{Z_\pm r}$ with $Z_\pm=-\tau e^{i\theta^\mp}$. Hence there holds
\begin{equation}
\begin{split}
\int_{\Gamma^\pm_\infty}u_0 \p_\nu u_{sing} d\sigma=& \ds K \int_0^\infty r^{\eta-1} [\phi'(\theta^+) u_0(r,\theta^+)-\phi'(\theta^-) u_0(r,\theta^-)] dr \\
\ds = & K \left( \phi'(\theta^+)I(Z_+,\eta)-\phi'(\theta^-)I(Z_-,\eta)\right),
\end{split}
\end{equation}
which together with the use of Proposition \ref{Propo_Upper}, $\phi(\theta)=\cos(\eta\theta+\Phi)$, readily yields \eqref{LowerBound}. 

The proof is complete. 
\end{proof}

\section{Proof of Theorem~\ref{main-theorem}}\label{PROOF}

One of the key ingredients to derive the stability theorem is the local decomposition of solutions to transmission problems in a neighborhood of each polygonal vertex. This is known from the pioneering work of Grisvard \cite{Grisvard} and from Kozlov, Maz'ya, Rossemann \cite{Kozlov} that the solution to a boundary value problem $\triangle u=f$ in a domain with a polygonal corner admits a decomposition in the form of (\ref{decom_sing}). This result is extended to transmission problems, and we refer to Kellogg \cite{Kellogg1,Kellogg2}, Dauge and Nicaise \cite{dauge1989oblique,nicaise1990polygonal}. We also refer to Bonnetier and Zhang \cite{bonnetier_zhang} for the characterization of the exponent $\eta$ in terms of the local geometric shape. Those results can be summarised as the following theorem. 

\begin{theorem}[Local decomposition of solutions to transmission problems]\label{theo_decom_sol}
Let $u\in H^1(\O)$ be the solution to (\ref{EQ}) with $D$ a polygon. We denote by $\mathcal{S}_D$ the set of vertices of $D$. Here the variables $r$, $\theta$ are related to the polar coordinates in the neighborhood of each vertex. Then the following decomposition holds,
\begin{equation}
u=u_{reg}+\sum_{x_i\in \mathcal{S}_D} K_i r^{\eta_i}\phi_i(\theta)\zeta_i,
\end{equation}
with the following proprieties,
\begin{enumerate}
\item $u_{reg}\in H^2(\O)$ and satisfies the same elliptic equation $\mathrm{div}[(1+(k-1)\chi_D)\n u] =0$.
\item $\Vert u_{reg} \Vert_{H^2(\O)}\leq C \Vert g \Vert_{H^{-1/2}(\p\O)}$ with $C$ independent of $u$.
\item The coefficient $K_i$ depends linearly on the data $g$.
\item The exponent $\eta_i\in (0,1)$ depends only on the conductivity $k$ and the geometry of the vertex.
\item $\phi_i$ is a piecewise smooth function depending only on the conductivity $k$ and the geometry of the vertex.
\item $\zeta_i$ is a smooth cut-off function such that $\zeta_i(r)=1$ if $r\leq\varrho_i$ and $\zeta_i(r)=0$ if $r\geq 2\varrho_i$. The radius $\varrho_i$ is chosen such that the disks $(B_{2\varrho_i}(x_i))_{i\in\mathcal{S}_D}$ do not intersect each other.
\end{enumerate}
\end{theorem}
With this theorem, we have the following useful results that are to be used for proving the main theorem of this paper.
\begin{enumerate}
\item For each vertex of a polygon $D\in \mathcal{D}$, the corresponding exponent $\eta$ satisfies $0<\eta_m\leq \eta\leq \eta_M<1$, with $\eta_m$, $\eta_M$ depending only on $k,a_m,a_M$.
\item It is possible to choose the data $g$ in a suitable space such that the coefficient $K$ can be bounded in $[1/K_0,K_0]$ with $K_0>0$ depending only on the a-priori parameters.
\item The $H^2$ norm of $u_{reg}$ depends only on the a-priori parameters and $\Vert g \Vert_{H^{-1/2}(\p \O)}$.
\item We choose for each cut-off function $\varrho=l/5$. Hence, the disks $(B_{2\varrho_i}(x_i))_{i\in\mathcal{S}_D}$ do not intersect each other.
\end{enumerate}  

We are ready to present the proof of Theorem~\ref{main-theorem}. 

\begin{proof}[Proof of Theorem \ref{main-theorem}]
Let $D, D'$ be two convex polygonal inclusions from the family $\mathcal{D}$ and let $Q$ be the convex hull of $D\cup D'$. We define the domain $\O'$ in Section \ref{Propagation} by, $\O':=\O\setminus Q$. Then, using Lemma \ref{lemme_geo}, $\O'$ is of Lipshitz class and the geometrical characteristics of $\p \O'$ depend only on the a-priori parameters, $\p\O$, $a_m$, $a_M$, $l$ and $\delta_0$. As an apparent consequence, the quantities like $r_0$, $r_m$ introduced in Section \ref{Propagation} depend only on those a-priori parameters.\par

Let $x_c$ be a vertex of $\p D$ such that $\h=\mathrm{dist}(x,D')$. We define the integral contours in Section \ref{Integral} with a radius $h=\min(\h/2,l/5,\d_0)$ and we conserve all notations in Section \ref{Integral}. With this radius, we can derive the following proprieties.
\begin{enumerate}
\item For any point $x'\in D'$, its distance to the ball $\B$ is at least $h$.
\item (\ref{decom_sing}) holds for the solution $u$ to (\ref{EQ}) and $\zeta(x)=1$ for any $x \in \B$.
\item All integral contours are supposed in $\B$, and the construction of $\p S^e$ requires $\tau \geq \tau_0=1/[2h\sin(\frac{\pi-b}{4})]=C_{\tau_0}h^{-1}$.
\end{enumerate}
Next we estimate $u-u'$ and $\n(u-u')$ on the contours $\p S^e$ and $\p S^i_Q$.\par 

It follows from Theorem \ref{theo_decom_sol} and its direct consequence (\ref{decom_sing}) that $u-u'$ is at least of class $\mathcal{C}^{\eta_m}$ in each neighborhood of the vertex of $Q$ with a norm $T_0>0$ that depends only on the a-priori data. Using the same arguments we can derive that the function $x\mapsto |x-x_c|\n (u-u')$ is also of class at least $\mathcal{C}^{\eta_m}$ with a norm denoted by $T_1$, which depends only on the a-priori data. Let $T=\max(T_0,T_1)$. We can then apply Proposition \ref{propagation_small} and Corollary \ref{propagation_boundary} with $\a=\eta_m$ and $\e < \e_m$, which is given by (\ref{critere_epsilon_h}). Then for all $x\in \p S^e$,
\begin{equation}\label{estimate_u-u'_e}
|u-u'|(x)\leq C_0 T (\ln|\ln \e|)^{-\eta_m},
\end{equation}
where $C_0$ is given in (\ref{valeur_C}). In what follows, we denote by $\d(\e)$ the quantity $(\ln|\ln \e|)^{-\eta_m}$. From the construction of $\p S^e$, we have $\frac{1}{2\tau} > \mathrm{dist}(x,\p Q)$ and thus for all $x\in \p S^e$,
\begin{equation}\label{estimate_gra_u-u'_e}
|\n(u-u')|(x)\leq 2C_0T \tau\d(\e).
\end{equation}
On the other hand, it follows directly from the Sobolev embedding that 
\begin{equation}\label{estimate_u-u'_i}
\Vert u-u'\Vert_{L^\infty(\p S^i_Q)} \leq C \Vert u-u' \Vert_{H^1(\O)} \leq C \Vert g \Vert_{H^{-1/2}(\p \O)}=C_{\p S^i,\infty},
\end{equation}
where the constant $C_{\p S^i,\infty}$ depends only on the a-priori data. Using Theorem \ref{theo_decom_sol}, we can obtain that,
\begin{eqnarray}\label{estimate_gra_u-u'_i}
&\ds \Vert \n(u-u')\Vert_{L^\infty(\p S^i_Q)} \leq 2K_0 h^{\eta_m-1}+\Vert u_{reg}\Vert_{H^2(\O)}+\Vert u'_{reg}\Vert_{H^2(\O)}\nonumber \\
&\leq 2K_0 h^{\eta_m-1}+ C_{\p S^i,1},
\end{eqnarray}
where the constant $C_{\p S^i,1}$ depends only on the a-priori data.\par

Then we apply Propositions \ref{Propo_Upper} and \ref{Propo_Lower} with the estimations (\ref{estimate_u-u'_e}), (\ref{estimate_gra_u-u'_e}), (\ref{estimate_u-u'_i}), and (\ref{estimate_gra_u-u'_i}). We absorb into the left hand side all the constants depending only on the a-priori parameters. There exist a constant $C$ depending only on a-priori parameters such that
\begin{equation}
C\tau^{-\eta}\leq \tau^{-\eta} e^{-\a'\tau h/2}+\tau^{-1}+h e^{-\a' \tau h}+ h e^{-\a' \tau h}(h^{\eta_m-1}+1+\tau)+h\tau \d(\e).
\end{equation}
Using the facts that $h\leq 1$ and $\tau \geq 1$ and the inequalities $e^{-x}\leq x^{-1}$, $e^{-x}\leq x^{-2}$ for all $x>0$, we have
\begin{equation}\label{eq:ddd1}
\begin{split}
C & \leq e^{-\a'\tau h/2}+ \tau^{\eta-1}+ h\tau^\eta e^{-\a' \tau h}+(h^{\eta_m}\tau^\eta+h\tau^\eta+h\tau^{1+\eta})e^{-\a' \tau h}+h\tau^{\eta+1}\d(\e), \\
& \leq h^{-1}\tau^{-1}+\tau^{\eta-1}+ \tau^{\eta-1}+h^{\eta_m-1}\tau^{\eta-1}+\tau^{\eta-1}+h^{-1}\tau^{\eta-1}+h\tau^{\eta+1}\d(\e),\\
& \leq h^{-1}\tau^{\eta-1}+h\tau^{\eta+1}\d(\e).
\end{split}
\end{equation}
We next determine a minimum modulo constants of the right hand side of the inequality in \eqref{eq:ddd1}. Set $\tau=\tau_e$ with
\begin{equation}
\tau_e=h^{-1}\d(\e)^{-1/2}.
\end{equation}
It is straightforward to verify that for $\e$ smaller than a certain constant one has that if
\begin{equation}\label{check_tau}
\d(\e)^{-1/2}\geq C_{\tau_0},
\end{equation}
then 
\begin{equation}\label{check_tau2}
\tau_e \geq \tau_0.
\end{equation}
Solving for $h$, it gives
\[h\leq C (\ln|\ln\e|)^{\frac{\eta_m(\eta-1)}{2\eta}},\]
and thus,
\begin{equation}\label{final}
\min(\h/2,l/5,\d_0) \leq C (\ln|\ln\e|)^{\frac{\eta_M-1}{2}}.
\end{equation}
Hence, if $\e$ is small enough such that $\e < \e_m$ in (\ref{critere_epsilon_h}), that (\ref{check_tau}) holds and that the right hand side of (\ref{final}) is smaller than $\min(l/5,\d_0)$, we have
\[\h \leq C (\ln|\ln\e|)^{\frac{\eta_M-1}{2}}.\]
Therefore, the claim of this theorem readily follows.

The proof is complete. 
\end{proof}

\section{On the uniqueness results}\label{Auxillaire}

Clearly, Theorem~\ref{main-theorem} implies the uniqueness of the inverse inclusion problem for polygonal inclusions under one measurement; see Remark~\ref{rem:1}. In this section, we further extend this kind of uniqueness result into a more general case where the conductivity $k\in L^\infty(\O)$ is a certain piecewise constant function. Next, we introduce the class of piecewise-constant conductivity functions within \textit{nested polygonal geometry}; see Fig.~\ref{geo_nested} for a schematic illustration. 

\begin{definition}
For $j\in \N^*:=\mathbb{N}\cup\{+\infty\}$, let $D_j\Subset \subset \O$ be a convex polygon such that
\[
D_{j+1}\Subset D_j.
\]
There exists $j_D\in\mathbb{N}$ such that $D_j=\emptyset$ when $j>j_D$. 
A conductivity function $k\in L^\infty(\O)$ is said to be \textit{piecewise constant within nested polygonal geometry} if there are constants $k_j >0$ with $k_{j+1}\neq k_j$, $k_1\neq 1$ such that
\begin{equation}\label{eq:ppp1}
k(x)=\sum_{j=1}^{\infty}k_j \chi_{\Sigma_j}(x),
\end{equation}
where $\Sigma_j=D_j \setminus \overline{D_{j+1}}$.
\end{definition}
\begin{figure}[!ht]
\includegraphics[scale=0.33]{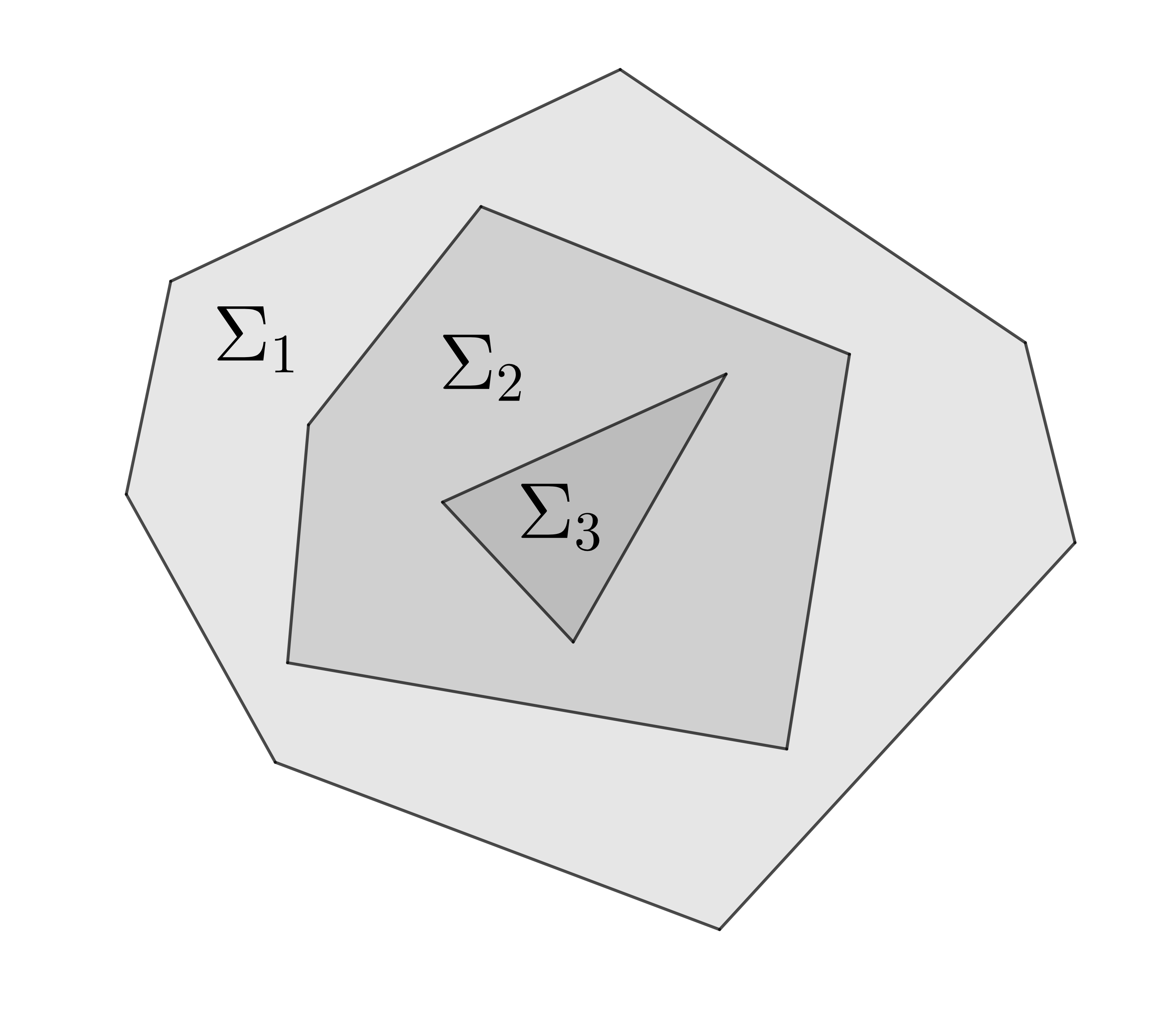}
\caption{Schematic illustration of a piecewise-constant conductivity function within nested polygonal geometry.}
\label{geo_nested}
\end{figure}

\begin{theorem}\label{unicite_geo_inclu}
Let $k,k' \in L^\infty(\O)$ be two piecewise-constant conductivity functions within nested polygonal geometry with the corresponding nested convex polygons being $(D_j)_{j\in\N^*}$, $(D'_j)_{j\in\N^*}$, respectively. Consider the conductivity problem \eqref{EQ} and let $u$ and $u'$ be the corresponding solutions associated with $k$ and $k'$, respectively. Suppose that the boundary input $g$ is admissible in the sense that \eqref{eq:condd1} is fulfilled for $u$ on each vertex of $D_j$ (respectively, $u'$ on each vertex of $D_j'$, $j\in\mathbb{N}^*$). 
Then if $u=u'$ on $\Gamma_0$, one must have $D_j=D'_j$ for all $j\in \N^*$ and $k=k'$.
\end{theorem}
\begin{proof}

We prove the theorem by induction. 

Set $k_0=k_0'=1$ and let $Q$ be the convex hull of $D_1\cup D'_1$. Clearly, one can construct a convex polygon $D_0$ such that $Q \Subset D_0 \Subset \O$. It is immediately seen from Theorem \ref{main-theorem} that $D_1=D'_1$. Using the unique continuation property, we have $u=u'$ in $\O\setminus D_1$ and thus $u=u'$, $\p_\nu u =\p_\nu u'$ on $\p D_0$. Then by setting $k_0=k'_0=1$ the case $m=1$ is true.\par

Suppose that for $j<m$ with $m\in\mathbb{N}$, there hold $D_j=D'_j$, $k_j=k'_j$, and $u=u'$, $\p_\nu u =\p_\nu u'$ on $\p D_j$. It follows from Theorem \ref{main-theorem} applied in the polygon $D_{m-1}$, and from the unique continuation property that $D_m=D'_m$ and $u=u'$, $\p_\nu u|_+=\p_\nu u'|_+$ on $\p D_m$. Next we prove $k_m=k'_m$.

Let $x_c\in \p D_m$ be a vertex of $D_m$. It follows from the Green formula and the jump relations that the following integral identity holds
\begin{equation}\label{IntId_k}
k_{m-1}(\frac{1}{k'_m}-\frac{1}{k_m})\int_{\Gamma^\pm}u_0\p_\nu u\ d\sigma=\int_{\p S^i_{D_m}} u_0 \p_\nu(u-u')-(u-u')\p_\nu u_0\ d\sigma, 
\end{equation}
for any harmonic function $u_0$ in a neighborhood of $x_c$. Here the value of $\p_\nu u$ is taken from the outside of $D_m$. Now we take $u_0$ to be the CGO solution constructed in (\ref{CGO}) and decompose $u$ in the left hand side using (\ref{decom_sing}). Then we can proceed the same analysis in Propositions \ref{Propo_Upper} and \ref{Propo_Lower}. It turns out there exists a constant $C\in\mathbb{R}_+$ independent of the choice of $u_0$ such that for all $\tau>0$,
\begin{equation}\label{eq:fff1}
C k_{m-1}\left|\frac{1}{k'_m}-\frac{1}{k_m}\right|\tau^{-\eta}\leq \tau^{-\eta}e^{-\a'\tau h/2}+\tau^{-1}+h e^{-\a' \tau h}+h (1+\tau)e^{-\a'\tau h}.
\end{equation}
Therefore, we can choose $\tau\in\mathbb{R}_+$ sufficiently small such that the left hand side is larger then the right hand side in \eqref{eq:fff1}. In doing so, one immediately has that 
\[k_m=k'_m.\]
Using the jump relations, we have $\p_\nu u|_-=\p_\nu u'|_-$ on $\p D_m$, which completes the induction. 

The proof is complete. 
\end{proof}

\section*{Acknowledgment}
The work was supported by the FRG and startup grants from Hong Kong Baptist University, Hong Kong RGC General Research Funds, 12302017 and 12301218.


\begin{thebibliography}{99}


\bibitem{alessandrini1988stable}
G. Alessandrini, {\it Stable determination of conductivity by boundary measurements}, {Appl. Anal.}, {\bf 27} (1988), 153--172.

\bibitem{gradu1}
G. Alessandrini, {\it An identification problem for an elliptic equation in two variables}, {Annali di matematica pura ed applicata}, {\bf 145(1)} (1986), {265--295}.

\bibitem{AMR} G. Alessandrini, M. de Hoop and R. Gaburro, {\it Uniqueness for the electrostatic inverse boundary value problem with piecewise constant anisotropic conductivities},
Inverse Problems, {\bf 33} (2017), no. 12, 125013. 


\bibitem{gradu2}
G. Alessandrini, and M. Rolando, {\it Elliptic equations in divergence form, geometric critical points of solutions, and Stekloff eigenfunctions}, {SIAM Journal on Mathematical Analysis}, {\bf 25(5)} (1994), {1259-1268}.

\bibitem{Alessandrini}
G. Alessandrini, L. Rondi, E. Rosset and S. Vessella,
{\it The stability for the Cauchy problem for elliptic equations}, {Inverse Problems}, {\bf 25} (2009),123004.

\bibitem{moi2}
H. Ammari, F. Triki and C.-H. Tsou, {\it Numerical determination of anomalies in multifrequency electrical impedance tomography}, {Euro. J. Appl. Math.} (2018), 1--24.

\bibitem{ammari2004reconstruction}
H. Ammari and H. Kang,
{\it Reconstruction of small inhomogeneities from boundary measurements}, {Springer}, 2004.

\bibitem{ammari2017identification}
H. Ammari and F. Triki,{\it Identification of an inclusion in multifrequency electric impedance tomography}, {Comm. PDEs}, {\bf 42} (2017), 159--177.

\bibitem{astala2006calderon}
K. Astala, and L. P{\"a}iv{\"a}rinta, {\it Calder{\'o}n's inverse conductivity problem in the plane}, {Ann. Math.}, (2006), 265--299.

\bibitem{barcelo1994inverse}
B. Barcel{\'o}, E. Fabes and J.-K. Seo, {\it The inverse conductivity problem with one measurement: uniqueness for convex polyhedra}, {Proc. AMS}, {\bf 122} (1994) 183--189.

\bibitem{beretta2019lipschitz}
E. Beretta, E. Francini and S. Vessella, {\it Lipschitz stability estimates for polygonal conductivity inclusions from boundary measurements}, {arXiv:1901.01152}

\bibitem{blaasten2017recovering}
E. Bl{\aa}sten and H. Liu, {\it Recovering piecewise constant refractive indices by a single far-field pattern}, { arXiv:1705.00815} 

\bibitem{Liu_stability_polygon_helmholtz}
E. Bl{\aa}sten and H. Liu, {\it On corners scattering stably and stable shape determination by a single far-field pattern}, { arXiv:1611.03647}.

\bibitem{blaasten2014corners}
E. Bl{\aa}sten, L. P{\"a}iv{\"a}rinta and J. Sylvester, {\it Corners always scatter}, {Comm. Math. Phys.}, {\bf 331} (2014), 725--753.


\bibitem{moi3}
E. Bonnetier, F. Triki and C.-H. Tsou, {\it On the electro-sensing of weakly electric fish}, {J. Math. Anal. Appl.}, {\bf 464} (2018), 280--303.

\bibitem{bonnetier_zhang}
E. Bonnetier and H. Zhang, {\it Characterization of the essential spectrum of the Neumann-Poincar\'e operator in 2d domains with corner via Weyl sequences}, {arXiv:1702.08127}.

\bibitem{CALDERON2006}
A.-P. Calder\'{o}n, {\it On an inverse boundary value problem}, { Computational \& Applied Mathematics}, {\bf 25} (2006), {133 -- 138}.

\bibitem{CSU}
K. E. Carlos, E. Sj\"ostrand and G. Uhlmann, {\it The Calder\'on problem with partial data}, Ann. of Math. (2), {\bf 165} (2007), no. 2, 567--591.

\bibitem{CHL} D. Choi, J. Helsing and M. Lim, {\it Corner effects on the perturbation of an electric potential}, SIAM J. Appl. Math., {\bf 78} (2018), 1577--1601.

\bibitem{dauge1989oblique}
M. Dauge and S. Nicaise, {\it Oblique derivative and interface problems on polygonal domains and networks}, {Comm. PDEs}, {\bf 14} (1989), {1147--1192}.

\bibitem{fabes1999inverse}
E. Fabes, H. Kang and J.-K. Seo, {\it Inverse conductivity problem with one measurement: Error estimates and approximate identification for perturbed disks}, {SIAM J. Math. Anal.}, {\bf 30} (1999), {699--720}.

\bibitem{friedman1989uniqueness}
A. Friedman and V. Isakov, {\it On the uniqueness in the inverse conductivity problem with one measurement}, {Indiana Univ. Math. J.}, {\bf 38}, (1989), {563--579}.

\bibitem{Grisvard}
P. Grisvard, {\it Elliptic Problems in Nonsmooth Domains}, {Society for Industrial and Applied Mathematics}, (2011).


\bibitem{IUY} O. Y. Imanuvilov, G. Uhlmann, M. Yamamoto, {\it The Calder\'on problem with partial data in two dimensions}, J. Amer. Math. Soc., {\bf 23} (2010), no. 3, 655--691.

\bibitem{Kang_ball_R3}
H. Kang and J.-K. Seo, {\it Inverse conductivity problem with one measurement: Uniqueness of balls in $\mathbb{R}^3$}, {SIAM J. Appl. Math.}, {\bf 59} (1999), {1533--1539}.

\bibitem{Kellogg1}
R.-B. Kellogg, {\it Singularities in interface problems}, {Numerical Solution of Partial Differential Equations--II}, {Elsevier}, 1971, {351--400}. 

\bibitem{Kellogg2}
R.-B. Kellogg, {\it Higher order singularities for interface problems}, {The mathematical foundations of the finite element method with applications to partial differential equations}, {Elsevier}, 1972, {589--602}.

\bibitem{Meyers}
J. Korevaar and J. L.-H. Meyers, {\it Logarithmic convexity for supremum norms of harmonic functions}, {Bulletin London Math. Soc.}, {\bf 26} (1994), {353--362}.

\bibitem{Kozlov}
V.-A. Kozlov, V.-G. Mazia and J. Rossmann, {\it Elliptic boundary value problems in domains with point singularities}, {American Mathematical Soc.}, {\bf 52} (1997).

\bibitem{Nachman1988}
A.-I. Nachman, {\it Reconstructions from boundary measurements},  {Ann. Math.}, {\bf 128} (1988), {531--576}.

\bibitem{nicaise1990polygonal}
S. Nicaise, {\it Polygonal interface problems: higher regularity results}, {Comm. PDEs}, {\bf 15} (1990), {1475--1508}.

\bibitem{SylvesterUhlmann1987}
J. Sylvester and G. Uhlmann, {\it A global uniqueness theorem for an inverse boundary value problem}, {Ann. Math.}, {\bf 125} (1987), {153--169}.

\bibitem{moi1}
F. Triki and C.-H. Tsou, {\it Inverse inclusion problem: A stable method to determine disks}, {HAL preprint, hal-01633360} (2017).



\end{thebibliography}
\end{document}